\documentclass[10pt,reqno]{amsart}
\usepackage{amsmath}
\usepackage{amssymb}
\usepackage{amsthm}
\usepackage[table,xcdraw]{xcolor}
\usepackage{booktabs}
\usepackage{multirow}
\usepackage{kotex}

\usepackage{scalerel}

\newcommand\reallywidehat[1]{\arraycolsep=0pt\relax%
\begin{array}{c}
\stretchto{
  \scaleto{
    \scalerel*[\widthof{\ensuremath{#1}}]{\kern-.5pt\bigwedge\kern-.5pt}
    {\rule[-\textheight/2]{1ex}{\textheight}} 
  }{\textheight} %
}{0.5ex}\\           
#1\\                 
\rule{-1ex}{0ex}
\end{array}
}

\newlength{\defbaselineskip}
\newcommand{\setlinespacing}[1]%
           {\setlength{\baselineskip}{#1 \defbaselineskip}}

\numberwithin{equation}{section}

\newtheorem{thm}{Theorem}[section]

\newtheorem{lem}[thm]{Lemma}

\theoremstyle{definition}

\theoremstyle{remark}

\numberwithin{equation}{section}

\usepackage{etoolbox}
\patchcmd{\subsubsection}{\itshape}{\bfseries}{}{}

\begin{document}

\title[The radius of spatial analyticity]
{On the radius of spatial analyticity for the Majda-Biello and Hirota-Satsuma systems}

\author{Seongyeon Kim and Ihyeok Seo}



\address{Department of Mathematics Education, Jeonju University, Jeonju 55069, Republic of Korea}
\email{sy\_kim@jj.ac.kr}

\address{Department of Mathematics, Sungkyunkwan University, Suwon 16419, Republic of Korea}
\email{ihseo@skku.edu}

\begin{abstract}
We investigate the persistence of spatial analyticity for solutions to the Majda-Biello and Hirota-Satsuma systems with analytic initial data. This result is the first to establish analyticity persistence in such coupled KdV systems.
\end{abstract}

\maketitle

\section{Introduction}\label{sec1}
In this paper, we are concerned with the persistence of spatial analyticity for the solutions of the following Majda-Biello and Hirota-Satsuma systems, given analytic initial data:

\medskip

\begin{itemize}
\item Majda-Biello system:
$$
\begin{cases}
u_t+u_{xxx}=-vv_x,\\
v_t+a_2v_{xxx}=-(uv)_x,\\
(u,v)|_{t=0}=(u_0,v_0),
\end{cases}
$$
where $a_2\neq0$. This system proposed by Majda and Biello in \cite{MB} serves as a reduced asymptotic model for the investigation of the nonlinear resonant interactions between long-wavelength equatorial Rossby waves and barotropic Rossby waves,

\medskip

\item Hirota-Satsuma system:
$$
\begin{cases}
u_t+a_1u_{xxx}=-6a_1uu_x+c_{12}vv_x,\\
v_t+v_{xxx}=-3uv_x,\\
(u,v)|_{t=0}=(u_0,v_0),
\end{cases}
$$
where $a_1\ne0$. Hirota and Satsuma introduced this system in \cite{HS} to describe the interaction between two long waves exhibiting distinct dispersion relations.
\end{itemize}

\medskip

\noindent
Here, $u(t,x)$ and $v(t,x)$ are real-valued unknown functions of the two variables $x,t\in\mathbb{R}$.
To treat both systems under a unified framework, we adopt the following formulation:
\begin{equation}\label{ssys}
\begin{cases}
u_t + a_1 u_{xxx} = c_{11}uu_x + c_{12}vv_x , \\
v_t + a_2 v_{xxx} =  c_{21}u_x v + c_{22}uv_x, \\
(u,v)|_{t=0} = (u_0,v_0),
\end{cases}
\end{equation}
where $a_1a_2\neq0$.
This system is said to be in \textit{divergence form} if $c_{21}=c_{22}$, and in \textit{non-divergence form} otherwise.
The Majda-Biello system is in divergence form and corresponds to a special case of \eqref{ssys} with the coefficients:
\begin{equation}\label{mb}
    a_1=1,\quad c_{11}=0,\quad c_{12}=-1,\quad c_{21}=c_{22}=-1.
\end{equation}
The Hirota-Satsuma system is in non-divergencs form and corresponds to a special case of \eqref{ssys} with the coefficients:
\begin{equation}\label{hs}
    a_2=1,\quad c_{11}=-6a_1,\quad c_{21}=0,\quad c_{22}=-3.
\end{equation}

While the well-posedness theory in Sobolev spaces is being actively studied, the spatial analyticity remains uncharted. Hereafter, our attention is directed towards the situation where we are presented real-analytic initial data with uniform radius of analyticity $\sigma_0>0$, so that there is a holomorphic extension to a complex strip
$$S_{\sigma_0}=\{x+iy:x,y\in\mathbb{R},\,|y|<\sigma_0\}.$$
Now, it is natural to ask whether this property may be continued analytically to a complex strip $S_{\sigma(t)}$ for all later times $t$, but with a possibly smaller and shrinking radius of analyticity $\sigma(t)>0$.

This class of problems was initially introduced by Kato and Masuda in their work \cite{KM}. It has since garnered significant attention in the context of various dispersive partial differential equations, including nonlinear Schr\"odinger equations (see, e.g., \cite{BGK2, Te1, AKS1}) and KdV equations (see, e.g., \cite{BG, BGK, SS, Te2, ST2, HW, PS, AKS2, ZLD, W, LW}). 
Demonstrating spatial analyticity for coupled equations, while simultaneously controlling all of them, poses a more challenging task. 
To the best of our knowledge, the study of spatial analyticity for such nonlinear dispersive systems has been limited to specific cases such as the Dirac-Klein-Gordon system in both one and two-dimensional spaces \cite{ST1, S1}, the Klein-Gordon-Schr\"odinger system \cite{AKS3}, a coupled system of modified KdV equations \cite{FH}, and the coupled system of BBM equations \cite{TTB}. However, there are currently no results addressing this issue for the Majda-Biello and Hirota-Satsuma systems. 
Motivated by this gap in the literature, our aim in this work is to establish the spatial analyticity for the systems.

Beyond merely filling this gap, it is also important to understand how spatial analyticity behaves in nontrivially coupled KdV-type systems, where nonlinear interactions between components introduce additional analytical difficulties. This calls for a systematic investigation of analyticity persistence in such coupled models.

In this work, we develop a unified analytic framework to study the persistence and long-time behavior of spatial analyticity for coupled KdV-type systems. In contrast to existing results obtained for specific models, the present approach provides a systematic treatment applicable to multiple coupled equations under a common structural formulation.

Among various coupled KdV-type systems, we consider the Majda–Biello and Hirota–Satsuma systems, both of which fall within the unified formulation \eqref{ssys}. These two models represent structurally distinct classes: the Majda–Biello system is in divergence form, whereas the Hirota–Satsuma system is in non-divergence form. They therefore serve as natural examples illustrating that the same analytic framework applies to both types of structures. Moreover, both systems arise in physically relevant contexts and admit a well-established well-posedness theory in Sobolev spaces, which makes them suitable for investigating  spatial analyticity in the global-in-time regime.

The Gevrey space, denoted $G^{\sigma,s}(\mathbb{R})$, $\sigma\ge0$, $s\in\mathbb{R}$, is a suitable function space to study analyticity of solution. In our case, it will be used with the norm
$$\|f\|_{G^{\sigma,s}}=\|e^{\sigma|D|}\langle D\rangle^sf\|_{L^2},$$
where $\langle D\rangle=1+|D|$ with $D=-i\partial_x$. According to the Paley-Wiener theorem\footnote{The proof given for $s=0$ in \cite{K} applies also for $s\in\mathbb R$ with some obvious modifications.} (see e.g. \cite{K}, p. 209), a function $f$ belongs to $G^{\sigma,s}$ with $\sigma>0$
if and only if it is the restriction to the real line of a function $F$ which is holomorphic in the strip $S_\sigma=\{x+iy:x,y\in\mathbb R,\,|y|<\sigma\}$ and satisfies $\sup_{|y|<\sigma}\|F(x+iy)\|_{H_x^s}<\infty$. Therefore every function in $G^{\sigma,s}$ with $\sigma>0$ has an analytic extension to the strip $S_\sigma$. 

Based on this property of the Gevrey space, which is the key to studying spatial analyticity of solution, our result below gives a lower bound $|t|^{-4/3-\epsilon}$ on the radius of analyticity $\sigma(t)$ of the solution to \eqref{ssys} as the time $t$ tends to infinity. Since we have two unknown functions in hand in the system, we define the space $\mathcal G^{\sigma,s}:=G^{\sigma,s}\times G^{\sigma,s}$ with $\|(f_1,f_2)\|_{\mathcal G^{\sigma,s}}:=\max\{\|f_1\|_{G^{\sigma,s}},\|f_2\|_{G^{\sigma,s}}\}$.

\begin{thm}\label{thm1}
Let $u,v$ be the global solution of \eqref{ssys} with $(u_0,v_0)\in\mathcal G^{\sigma_0,s}(\mathbb R)$ for some $\sigma_0>0$ and $s\in\mathbb R$. If the coefficients $c_{ij}$ $(i,j=1,2)$ fall into any of the cases vary depending on the value of $a_2/a_1$ as in Table \ref{table3}, then for all $t\in\mathbb R$
$$\left(u(t),v(t)\right)\in \mathcal G^{\sigma(t),s}(\mathbb R)$$
with $\sigma(t)\ge c|t|^{-4/3-\epsilon}$ for any $\epsilon>0$ as $|t|\rightarrow\infty$.
Here, $c>0$ is a constant depending on $\|u_0\|_{G^{\sigma_0,s}(\mathbb R)}$ and $\|v_0\|_{G^{\sigma_0,s}(\mathbb R)}$.
\end{thm}

It follows from \eqref{mb} and \eqref{hs} that this theorem applies to a wide range of physically relevant parameter regimes:
\begin{itemize}
    \item \textbf{Majda–Biello system:} when $a_2 < 0$, $a_2 = 1$, or $a_2 > 4$.
    \item \textbf{Hirota–Satsuma system:} when $a_1 < \tfrac{1}{4}$ ($a_1\neq0$).
\end{itemize}

It should be also noted that the existence of the global solution in the theorem is guaranteed as long as the Cauchy problem \eqref{ssys} is globally well-posed in $H^s$ for some $s\in\mathbb{R}$.\footnote{The well-posedness of the coupled KdV-KdV systems with initial data in Sobolev spaces $H^s(\mathbb R)\times H^s(\mathbb R)$ has been studied by many authors. For the latest result as well as brief history, we refer the reader to \cite{YZ} and the references therein. In particular, the Majda-Biello system is globally well-posed for $s>-3/4$ if $a_2=1$, for $s\ge1$ if $a_2= 4$, and for $s\ge0$ in the other cases; the Hirota-Satsuma system for $s\ge1$ if $a_1=1/4$ and $c_{12}>0$, and for $s\ge0$ if $a_1\notin\{1/4,1\}$ and $c_{12}>0$.}
Indeed, observe that $G^{0,s}$ coincides with the Sobolev space $H^s$ and the embeddings
\begin{equation}\label{emb}
G^{\sigma,s}\subset G^{\sigma^\prime,s^\prime},\ \mathcal G^{\sigma,s}\subset \mathcal G^{\sigma^\prime,s^\prime}
\end{equation}
hold for all $0\le\sigma'<\sigma$ and $s,s'\in\mathbb{R}$. This in turn signifies that, with $\sigma'=0$ and the assumption that \eqref{ssys} is globally well-posed in $H^{s'}(=G^{0,s'})$ for some $s'\in\mathbb{R}$, \eqref{ssys} has a unique global solution, given initial data $(u_0,v_0)\in\mathcal G^{\sigma_0,s}$ for any $\sigma_0>0$ and $s\in\mathbb R$.

Therefore, spatial analyticity is meaningful only when the existence of a global-in-time solution is guaranteed. The specific ranges of the parameters appearing in the main results and auxiliary lemmas of this paper are naturally constrained by the currently known global well-posedness theory for the system \eqref{ssys}. In particular, they precisely cover the 
range in which the existence of global solutions is known in Sobolev spaces.

\begin{table}
\setlength{\tabcolsep}{7pt}
\renewcommand{\arraystretch}{1.5}
\begin{tabular}{|c|c|}
\hline
\rowcolor[HTML]{EFEFEF} 
\multicolumn{1}{|c|}{\cellcolor[HTML]{EFEFEF}$\frac{a_2}{a_1}$} & \multicolumn{1}{c|}{\cellcolor[HTML]{EFEFEF}Coefficients $c_{ij}$} \\ \hline
$\frac{a_2}{a_1}<0$                   & arbitrary                \\ \hline
$0<\frac{a_2}{a_1}\neq1\leq4$ & $c_{12}=c_{21}=c_{22}=0$          \\ \hline
$\frac{a_2}{a_1}=1$                   & $c_{21}=c_{22}$           \\ \hline
$\frac{a_2}{a_1}>4$                   & arbitrary                 \\ \hline
\end{tabular}
\vspace{0.5cm}
\caption{The coefficients $c_{ij}$ $(i,j=1,2)$ vary depending on the range of $a_2/a_1$.}\label{table3}
\end{table}

\medskip

The structure of this paper is outlined as follows: In Section \ref{sec2}, we introduce essential function spaces such as Bourgain and Gevrey-Bourgain spaces, along with their basic properties, which will be utilized in subsequent sections.
Section \ref{sec3} is dedicated to presenting bilinear estimates in Gevrey-Bourgain spaces. By employing a contraction argument based on these estimates, we establish that for a short time interval $0 \leq t \leq \delta$, where $\delta > 0$ depends on the norm of the initial data, the radius of analyticity remains strictly positive. Additionally, we demonstrate an approximate conservation law, even though the exact conservation of the $\mathcal G^{\sigma_0,s}$-norm of the solution does not hold. This approximation is crucial for controlling the solution's growth within the time interval $[0, \delta]$, measured in the data norm $\mathcal G^{\sigma_0,s}$.
In Section \ref{sec4}, we delve into the proofs of the local result and the almost conservation law.
Section \ref{sec5} is dedicated to finalizing the proof of Theorem \ref{thm1} by iterating the local result based on the conservation law.

\medskip

Throughout this paper, 
we write $A\lesssim B$ to mean $A\le CB$ for a positive constant $C$, and $A\sim B$ to mean $B\lesssim A\lesssim B$.

\section{Function spaces}\label{sec2}
In this section, we present certain function spaces along with their basic properties. These spaces will be employed in subsequent sections to prove Theorem \ref{thm1}.

For $s,b\in\mathbb R$ and a polynomial $p$ on $\mathbb{R}$, $X_p^{s,b}=X_p^{s,b}(\mathbb R^{1+1})$ denotes the Bourgain space defined by the norm
$$\|f\|_{X_p^{s,b}}=\| \langle\xi\rangle^s\langle\tau-p(\xi)\rangle^b
\hat f(\tau,\xi)\|_{L^2_{\tau,\xi}},$$
where $\langle\cdot\rangle=1+|\cdot|$ and $\hat f$ is the space-time Fourier transform of $f$ given by
$$\hat f(\tau,\xi)=\int_{\mathbb R^{2}}e^{-i(t\tau+x\xi)}f(t,x)dtdx.$$
For simplicity, we write $X_u^{s,b}$ and $X_v^{s,b}$ for $X_p^{s,b}$ when $p(\xi)=a_1\xi^3$ and $p(\xi)=a_2\xi^3$, respectively.
The restriction of the Bourgain space, denoted $X_{p,\delta}^{s,b}$, to a time slab $(0,\delta )\times\mathbb R$
is a Banach space when equipped with the norm
$$\|f\|_{X_{p,\delta}^{s,b}}=\inf\big\{\|g\|_{X_p^{s,b}}:g=f\ \text{on}\ (0,\delta)\times\mathbb R\big\}.$$
We also introduce the Gevrey-Bourgain space\footnote{Bourgain \cite{B} previously employed the Gevrey-modification of Bourgain spaces to investigate the persistence of analyticity in solutions of the Kadomtsev-Petviashvili equation. He demonstrated that the radius of analyticity remains positive for the entire duration of the solution's existence. This argument possesses generality and applies to a class of dispersive equations; however, it does not yield a lower bound on the radius $\sigma(t)$ as $|t|\rightarrow\infty$.} $X_p^{\sigma,s,b}=X_p^{\sigma,s,b}(\mathbb{R}^{1+1})$ defined by the norm
$$
\|f\|_{X_p^{\sigma,s,b}}=\|e^{\sigma| \partial_x|}f\|_{X_p^{s,b}}.
$$
Its restriction $X_{p,\delta}^{\sigma,s,b}$ to a time slab $(0,\delta)\times\mathbb R$ is defined in a similar way as earlier, and when $\sigma=0$, $X_p^{\sigma,s,b}$ and $X_{p,\delta}^{\sigma,s,b}$ coincide with $X_p^{s,b}$ and $X_{p,\delta}^{s,b}$, respectively. We define the space $\mathcal{X}_{(u,v)}^{s,b}:=X_u^{s,b}\times X_v^{s,b}$ with the norm $\|(f_1,f_2)\|_{\mathcal{X}_{(u,v)}^{s,b}}:=\max\{\|f_1\|_{X_u^{s,b}},\|f_2\|_{X_v^{s,b}}\}$, and we do similar to all the derived forms of the Bourgain space analogously.

We now list some basic properties of these spaces. For the case where $\sigma=0$, the proofs of the first two lemmas presented below can be found in Section 2.6 of \cite{T1}, while the third lemma can be derived using the same argument as that used in Lemma 3.1 of \cite{CKSTT}. Through the substitution $f\rightarrow e^{\sigma|\partial_x|}f$, the properties of $X_p^{s,b}$ and its restrictions can be extended to $X_p^{\sigma,s,b}$.

\begin{lem}\label{lem00}
Let $\sigma\ge0$, $s\in\mathbb R$ and $b>1/2$. Then, $X_p^{\sigma,s,b}\subset C(\mathbb R,G^{\sigma,s})$ and
$$\sup_{t\in\mathbb R}\|f(t)\|_{G^{\sigma,s}}\le C\|f\|_{X_p^{\sigma,s,b}},$$
where $C>0$ is a constant depending only on $b$.
\end{lem}

\begin{lem}\label{lem2}
Let $\sigma\ge0$, $s\in\mathbb R$, $-1/2<b<b'<1/2$ and $\delta>0$. Then
$$\|f\|_{X_{p,\delta}^{\sigma,s,b}}\le C\delta^{b^\prime-b}\|f\|_{X_{p,\delta}^{\sigma,s,b^\prime}},$$
where $C>0$ is a constant depending only on $b$ and $b'$.
\end{lem}

\begin{lem}\label{lem3}
Let $\sigma \ge 0$, $s\in\mathbb R$, $-1/2<b<1/2$ and $\delta>0$. Then, for any time interval $I\subset[0,\delta]$,
$$\|\chi_If\|_{X_p^{\sigma,s,b}}\le C\|f\|_{X_{p,\delta}^{\sigma,s,b}},$$
where $\chi_I(t)$ is the characteristic function of $I$, and $C>0$ is a constant depending only on $b$.
\end{lem}

\section{Bilinear estimates in Gevrey-Bourgain spaces}\label{sec3}

In this section, we present bilinear estimates (Lemma \ref{bilinear}) 
in the Gevrey-Bourgain spaces.
Using these estimates, we derive Lemma \ref{f}, which plays a crucial role in establishing the almost conservation law in the next section.

\begin{lem}\label{bilinear}
Let $s>-3/4$.
Then there exist $b\in(1/2,1)$ and $\epsilon>0$ such that some of the following estimates, depending on the range of $a_2/a_1$, as shown in Table \ref{table},
hold true for all $\sigma \ge 0$ and all $b' \in [b, b + \epsilon)$:
\begin{align}
&\|(fg)_x\|_{X_u^{\sigma,s,b'-1}}\le C\|f\|_{X_u^{\sigma,s,b}}\|g\|_{X_u^{\sigma,s,b}},\label{comm0}\\
&\|(fg)_x\|_{X_u^{\sigma,s,b'-1}}\le C\|f\|_{X_v^{\sigma,s,b}}\|g\|_{X_v^{\sigma,s,b}},\label{comm1}\\
&\|(fg)_x\|_{X_v^{\sigma,s,b'-1}}\le C\|f\|_{X_u^{\sigma,s,b}}\|g\|_{X_v^{\sigma,s,b}},\label{comm2}\\
&\|fg_x\|_{X_v^{\sigma,s,b'-1}}\le C\|f\|_{X_v^{\sigma,s,b}}\|g\|_{X_u^{\sigma,s,b}},\label{comm3}\\
&\|fg_x\|_{X_v^{\sigma,s,b'-1}}\le C\|f\|_{X_u^{\sigma,s,b}}\|g\|_{X_v^{\sigma,s,b}}.\label{comm4}
\end{align}
Here, $C>0$ is a constant depending only on $s$, $b$ and $b'$.
\end{lem}

The estimate \eqref{comm0} can be found in Corollary 10 of \cite{SS}. 
The remaining estimates \eqref{comm1}–\eqref{comm4} correspond respectively to 
the estimates (3.4), (3.5), (3.7), and (3.6) in \cite[Theorem 3.5]{YZ} by Yang and Zhang. In their work, \eqref{comm1}–\eqref{comm4} are proved under the assumption $b' = b$, in the context of establishing local well-posedness in Sobolev spaces. However, their argument can be readily adapted to the more general setting $b' \in [b, b + \epsilon)$. We therefore omit the details here.

Instead, we derive the following important lemma from Lemma \ref{bilinear}, in which the nonlinear functions $f_1(u,v)$ and $f_2(u,v)$ are defined as follows:
$$
f_1(u,v)=\frac{c_{11}}2\partial_x\big(e^{\sigma|\partial_x|}u^2-(e^{\sigma|\partial_x|}u)^2\big)+\frac{c_{12}}2\partial_x\big(e^{\sigma|\partial_x|}v^2-(e^{\sigma|\partial_x|}v)^2\big)
$$
and
$$
f_2(u,v)=c_{21}\big(e^{\sigma|\partial_x|}(u_x v)-(e^{\sigma|\partial_x|}u_x)(e^{\sigma|\partial_x|}v)\big)+c_{22}\big(e^{\sigma|\partial_x|}(uv_x)-(e^{\sigma|\partial_x|}u)(e^{\sigma|\partial_x|}v_x)\big),
$$
which will be introduced in Section \ref{sec4} (see \eqref{f1}, \eqref{f2}).

\begin{table}
\setlength{\tabcolsep}{2pt}
\renewcommand{\arraystretch}{1.5}
\begin{tabular}{|c|c|c|c|c|c|}
\hline
\rowcolor[HTML]{EFEFEF} 
\multicolumn{1}{|c|}{\cellcolor[HTML]{EFEFEF}} &
  \multicolumn{1}{c|}{\cellcolor[HTML]{EFEFEF}$\frac{a_2}{a_1}<0$} &
   $0< \frac{a_2}{a_1}\neq1\leq 4$ &
 $\frac{a_2}{a_1}=1$&
  $\frac{a_2}{a_1}>4$ \\ \hline
$s>-\frac34$ &
  \begin{tabular}[c]{@{}l@{}}\eqref{comm0}, \eqref{comm1}, \eqref{comm2}\\ \eqref{comm3}, \eqref{comm4}\end{tabular} &
  \eqref{comm0} &
  \eqref{comm0}, \eqref{comm1}, \eqref{comm2} &
  \begin{tabular}[c]{@{}l@{}}\eqref{comm0}, \eqref{comm1}, \eqref{comm2}, \\\eqref{comm3}, \eqref{comm4} 
  \end{tabular}\\ \hline
\end{tabular}
\vspace{0.5cm}
\caption{The bilinear estimates vary depending on the range of $a_2/a_1$.}\label{table}
\end{table}

\begin{lem}\label{f}
Given $\rho\in[0,3/4)$,
there exist $b\in(1/2,1)$ and $C_{\rho,b}>0$ such that
\begin{equation}\label{festimate}
\|f_1(u,v)\|_{X_u^{0,b-1}}+\|f_2(u,v)\|_{X_v^{0,b-1}}\le C_{\rho,b}\sigma^\rho\|(u,v)\|^2_{\mathcal X_{(u,v)}^{\sigma,0,b}}
\end{equation}
for all $\sigma\ge0$, and the coefficients $c_{ij}$ $(i,j=1,2)$ vary depending on the range of $a_2/a_1$ as in Table \ref{table3}.
\end{lem}

\begin{proof}
Note first that the bilinear forms $f_1$ and $f_2$ in terms of $(u, v)$ are constructed as follows: The first term in $f_1$ represents a bilinear form involving $(u, u)$, while the second term involves a bilinear form with $(v, v)$. The first term in $f_2$ encompasses a bilinear form with $(u_x, v)$, while the second term involves a bilinear form with $(u, v_x)$. 
In view of this observation, 
we will employ the bilinear estimates in Lemma 
\ref{bilinear} to estimate the bilinear forms. 

The crucial observation in this process is that, as evident from Table \ref{table}, the applicable estimates vary depending on the range of $a_2/a_1$, consequently leading to conditions on the coefficients $c_{ij}$ ($1\le i,j\le2$) as depicted in Table \ref{table3}.

When $a_2/a_1<0$ or $a_2/a_1>4$, all the bilinear estimates \eqref{comm0}-\eqref{comm4} become applicable. Hence, we can control all the bilinear terms in $f_1$ and $f_2$, and there is no need to impose constraints on the coefficients.

When $0<\frac{a_2}{a_1}\neq 1\leq 4$, 
we can utilize only \eqref{comm0}, allowing us to estimate solely the first term in $f_1$.
Consequently, we are unable to control all the bilinear terms in $f_1$ and $f_2$ with respect to $c_{12}$, $c_{21}$, and $c_{22}$, resulting in the relevant constants being zero as in Table \ref{table3}.

In the last case where $a_2/a_1=1$, the applicable estimates (as shown in Table \ref{table}) encompass \eqref{comm0}, \eqref{comm1}, and \eqref{comm2}.
The first two estimates control all the bilinear terms in $f_1$, but
the last estimate is only valuable when dealing with bilinear terms resembling a divergence form, such as $(uv)_x$. Note that the two terms in $f_2$ combine to form a divergence form,
\begin{equation}\label{foot}
c_{21}\partial_x\big(e^{\sigma|\partial_x|}(uv)-(e^{\sigma|\partial_x|}u)(e^{\sigma|\partial_x|}v)\big),
\end{equation}
provided $c_{21}=c_{22}$ as in Table \ref{table3}.
This condition aligns with the system \eqref{ssys} being in a divergence form.

Now we will show explicitly how to employ the bilinear estimates to achieve the desired estimation of $f_1$ and $f_2$, as in \eqref{festimate}.
Note first that
\begin{align*}
\widehat{g_1g_2}(\tau,\xi)=(\widehat{g_1}*\widehat{g_2})(\tau,\xi)
&=\int_{\mathbb R^{2}}\widehat{g_1}(\tau_1,\xi_1)\widehat{g_2}(\tau-\tau_1,\xi-\xi_1)d\tau_1d\xi_1\\
&=\int_{\mathbb R^{2}}\widehat{g_1}(\tau_1,\xi_1)\widehat{g_2}(\tau_2,\xi_2)d\tau_1d\xi_1
\end{align*}
with $\tau_2=\tau-\tau_1$ and $\xi_2=\xi-\xi_1$.
By the definition of the $X_u^{s,b}$ norm, we then see
\begin{align*}
\|f_1(u,v)\|_{X_u^{0,b-1}}
&\lesssim|c_{11}|\Big\|\frac{\xi}{\langle\tau-a_1\xi^3\rangle^{1-b}}\int_{\mathbb R^{2}}(e^{\sigma|\xi_1|}e^{\sigma|\xi_2|}-e^{\sigma|\xi|})\widehat u(\tau_1,\xi_1)\widehat u(\tau_2,\xi_2)d\tau_1d\xi_1\Big\|_{L^2_{\tau,\xi}}\\
&+|c_{12}|\Big\|\frac{\xi}{\langle\tau-a_1\xi^3\rangle^{1-b}}\int_{\mathbb R^{2}}(e^{\sigma|\xi_1|}e^{\sigma|\xi_2|}-e^{\sigma|\xi|})\widehat v(\tau_1,\xi_1)\widehat v(\tau_2,\xi_2)d\tau_1d\xi_1\Big\|_{L^2_{\tau,\xi}},
\end{align*}
and
\begin{align*}
\|f_2(u,v)\|_{X_v^{0,b-1}}
&\lesssim|c_{21}|\Big\|\frac{1}{\langle\tau-a_2\xi^3\rangle^{1-b}}\int_{\mathbb R^{2}}(e^{\sigma|\xi_1|}e^{\sigma|\xi_2|}-e^{\sigma|\xi|})\xi_1\widehat u(\tau_1,\xi_1)\widehat v(\tau_2,\xi_2)d\tau_1d\xi_1\Big\|_{L^2_{\tau,\xi}}\\
\label{bif}
&+|c_{22}|\Big\|\frac1{\langle\tau-a_2\xi^3\rangle^{1-b}}\int_{\mathbb R^{2}}(e^{\sigma|\xi_1|}e^{\sigma|\xi_2|}-e^{\sigma|\xi|})\widehat u(\tau_1,\xi_1)\xi_2\widehat v(\tau_2,\xi_2)d\tau_1d\xi_1\Big\|_{L^2_{\tau,\xi}}.
\end{align*}
In particular, when the two terms in $f_2$ can combine to take on a divergence form such as in the case where $a_2/a_1=1$, as explained in \eqref{foot}, the corresponding bound is replaced by a single term:
\begin{equation}\label{divcase}
    |c_{21}|\Big\|\frac{\xi}{\langle\tau-a_1\xi^3\rangle^{1-b}}\int_{\mathbb R^{2}}(e^{\sigma|\xi_1|}e^{\sigma|\xi_2|}-e^{\sigma|\xi|})\widehat u(\tau_1,\xi_1)\widehat v(\tau_2,\xi_2)d\tau_1d\xi_1\Big\|_{L^2_{\tau,\xi}}.
    \end{equation}

Next, we use the following inequality\footnote{This follows by combining
$e^{\sigma|\xi_1|}e^{\sigma|\xi_2|}-e^{\sigma|\xi_1+\xi_2|}\le\big(2\sigma \min\{|\xi_1|,|\xi_2|\}\big)^\rho e^{\sigma|\xi_1|}e^{\sigma|\xi_2|}$,
which can be found in \cite[Lemma 12]{SS}, and 
$$\min\{|\xi_1|,|\xi_2|\}\le\frac{2(1+\max\{|\xi_1|,|\xi_2|\})(1+\min\{|\xi_1|,|\xi_2|\})}{2+2\max\{|\xi_1|,|\xi_2|\}}
\leq \frac{2\langle\xi_1\rangle\langle\xi_2\rangle}{1+|\xi_1|+|\xi_2|}\leq \frac{2\langle\xi_1\rangle\langle\xi_2\rangle}{\langle\xi_1+\xi_2\rangle}.$$}
\begin{equation*}
e^{\sigma|\xi_1|}e^{\sigma|\xi_2|}-e^{\sigma|\xi_1+\xi_2|}\lesssim_\rho \frac{\sigma^{\rho}\langle\xi_1\rangle^\rho\langle\xi_2\rangle^\rho}{\langle\xi_1+\xi_2\rangle^\rho}e^{\sigma|\xi_1|}e^{\sigma|\xi_2|},
\end{equation*}
where $\sigma\ge0$, $0\le\rho\le 1$ and $\xi_1,\xi_2 \in \mathbb{R}$,
to arrive at
\begin{equation}\label{bif3}
\begin{aligned}
\|f_1(u,v)\|_{X_u^{0,b-1}}
\lesssim_{\rho}&|c_{11}|\sigma^{\rho}\Big\|\partial_x\big(e^{\sigma|\partial_x|}\langle\partial_x\rangle^\rho |\widehat u|^{\vee}\cdot
e^{\sigma|\partial_x|}\langle\partial_x\rangle^\rho |\widehat u|^{\vee}\big)\Big\|_{X_u^{0,-\rho,b-1}}\\
&+|c_{12}|\sigma^\rho\Big\|\partial_x\big(e^{\sigma|\partial_x|}\langle\partial_x\rangle^\rho |\widehat v|^{\vee}\cdot
e^{\sigma|\partial_x|}\langle\partial_x\rangle^\rho |\widehat v|^{\vee}\big)\Big\|_{X_u^{0,-\rho,b-1}},
\end{aligned}
\end{equation}
and
\begin{equation}\label{bif32}
\begin{aligned}
\|f_2(u,v)\|_{X_v^{0,b-1}}
\lesssim_{\rho}&|c_{11}|\sigma^\rho\Big\|e^{\sigma|\partial_x|}\langle\partial_x\rangle^\rho \partial_x w_1 \cdot
e^{\sigma|\partial_x|}\langle\partial_x\rangle^\rho |\widehat v|^{\vee}\Big\|_{X_v^{0,-\rho,b-1}}\\
&+|c_{12}|\sigma^\rho\Big\|e^{\sigma|\partial_x|}\langle\partial_x\rangle^\rho |\widehat u|^{\vee}\cdot
e^{\sigma|\partial_x|}\langle\partial_x\rangle^\rho \partial_x w_2\Big\|_{X_v^{0,-\rho,b-1}}
\end{aligned}
\end{equation}
with $w_1,w_2$ such that $\partial_x w_1=|\widehat{u_x}|^{\vee}$ and $\partial_x w_2 =|\widehat{v_x}|^{\vee}$.
The right-hand side of \eqref{bif32} is replaced by a single term,
$$|c_{21}|\sigma^\rho\Big\|\partial_x\big(e^{\sigma|\partial_x|}\langle\partial_x\rangle^\rho |\widehat u|^{\vee}\cdot
e^{\sigma|\partial_x|}\langle\partial_x\rangle^\rho |\widehat v|^{\vee}\big)\Big\|_{X_v^{0,-\rho,b-1}},
$$
corresponding to \eqref{divcase}.

We finally apply the bilinear estimates in Lemma \ref{bilinear}, with $\sigma=0$, $s=-\rho\in(-\frac34,0]$ and $b'=b\in(1/2,1)$, to \eqref{bif3} and \eqref{bif32} in each range of $a_2/a_1$ as explained at the beginning of the proof. 
As a result, for a given $\rho\in[0,3/4)$, there exists $b\in(1/2,1)$ such that
\begin{align*}
\|f_1(u,v)\|_{X_u^{0,b-1}}+\|f_2(u,v)\|_{X_v^{0,b-1}} &\lesssim_{\rho,b}\sigma^{\rho}\big\|e^{\sigma|\partial_x|}\langle\partial_x\rangle^{\rho}u\big\|_{X_u^{0,-\rho,b}}^2+\sigma^{\rho}\big\|e^{\sigma|\partial_x|}\langle\partial_x\rangle^{\rho}v\big\|_{X_v^{0,-\rho,b}}^2\\
&\qquad\qquad+\sigma^\rho\big\|e^{\sigma|\partial_x|}\langle\partial_x\rangle^{\rho}u\big\|_{X_u^{0,-\rho,b}}\big\|e^{\sigma|\partial_x|}\langle\partial_x\rangle^{\rho}v\big\|_{X_v^{0,-\rho,b}}\\
&\lesssim_{\rho,b} \sigma^\rho\|(u,v)\|_{\mathcal X_{(u,v)}^{\sigma,0,b}}^2
\end{align*}
for all $\sigma\ge0$, as desired.
\end{proof}

\section{Local well-posedness and almost conservation law}\label{sec4}
In this section we shall establish the local well-posedness in Subsection \ref{subsec4.1} and the almost conservation law in Subsection \ref{subsec4.2}.

\subsection{Local well-posedness}\label{subsec4.1}
We establish the following local well-posedness result in the space $\mathcal G^{\sigma,s}$ using Picard iteration in the $X_{p,\delta}^{\sigma,s,b}$ framework, together with Lemma \ref{lem00}. The solution exists on a time interval $0\le t\le\delta$ for some $\delta>0$ depending on the size of the initial data. In other words, we show that the radius of analyticity remains strictly positive throughout this short time interval.

\begin{thm}\label{thm2}
Let $\sigma>0$ and $s>-3/4$. Then for any initial data $(u_0,v_0)\in\mathcal G^{\sigma,s}$, there exist a time $\delta>0$ and a unique solution $(u,v)\in C([0,\delta],\mathcal G^{\sigma,s})$ to the Cauchy problem \eqref{ssys}, where the coefficients $c_{ij}$ $(i,j=1,2)$ are chosen based on the ratio $a_2/a_1$ as specified in Table \ref{table3}. 
The existence time $\delta$ can be taken as
\begin{equation}\label{deldel}
\delta = c_0(1+\|u_0\|_{G^{\sigma,s}}+\|v_0\|_{G^{\sigma,s}})^{-a}
\end{equation}
for some constants $c_0>0$ and $a>2$ depending only on $s$.
Furthermore, for any $1/2<b<1$, the solution $(u,v)$ satisfies
\begin{equation}\label{lowe}
\|u\|_{X_{u,\delta}^{\sigma,s,b}} \le C(\|u_0\|_{G^{\sigma,s}} + \|v_0\|_{G^{\sigma,s}}),\quad
\|v\|_{X_{v,\delta}^{\sigma,s,b}} \le C(\|u_0\|_{G^{\sigma,s}} + \|v_0\|_{G^{\sigma,s}})
\end{equation}
for some constant $C>0$ depending only on $b$.
\end{thm}

\begin{proof}
By Lemma \ref{lem00}, we work in the space $X^{\sigma,s,b}_\delta$ with $\sigma\ge0$, $s\in\mathbb{R}$, and $b>1/2$ instead of directly using $G^{\sigma,s}$. We begin by considering the linearized version of the coupled KdV–KdV system:
$$\begin{cases}
u_t + a_1 u_{xxx} = F_1(t,x), \\
v_t + a_2 v_{xxx} = F_2(t,x), \\
u(0,x) = u_0(x),\quad v(0,x) = v_0(x).
\end{cases}$$
Using Duhamel's principle, the solution is given by
\begin{equation}\label{DF}
\begin{cases}
u(t,x) = e^{-ta_1\partial_x^3}u_0(x) + \int_0^t e^{-(t-t')a_1\partial_x^3}F_1(t',\cdot)\,dt', \\
v(t,x) = e^{-ta_2\partial_x^3}v_0(x) + \int_0^t e^{-(t-t')a_2\partial_x^3}F_2(t',\cdot)\,dt',
\end{cases}
\end{equation}
where the Fourier multiplier $e^{itp(-i\partial_x)}$ is defined as
$$e^{itp(-i\partial_x)}f(x) = \frac{1}{2\pi} \int_{\mathbb{R}} e^{ix\xi} e^{itp(\xi)} \hat{f}(\xi)\,d\xi.$$

We now state the standard energy estimates in the $X_{p,\delta}^{\sigma,s,b}$ space (see Proposition 2.12 in \cite{T1}):

\begin{lem}\label{lem1}
Let $\sigma\ge0$, $s\in\mathbb{R}$, $1/2<b\le1$, and $0<\delta\le1$. Then the following estimates hold:
$$\|e^{-ta_1\partial_x^3}u_0\|_{X_{u,\delta}^{\sigma,s,b}} \le C\|u_0\|_{G^{\sigma,s}},\quad
\left\|\int_0^t e^{-(t-t')a_1\partial_x^3}F_1(t',\cdot)\,dt'\right\|_{X_{u,\delta}^{\sigma,s,b}} \le C\|F_1\|_{X_{u,\delta}^{\sigma,s,b-1}},$$
$$
\|e^{-ta_2\partial_x^3}v_0\|_{X_{v,\delta}^{\sigma,s,b}} \le C\|v_0\|_{G^{\sigma,s}},
\quad
\left\|\int_0^t e^{-(t-t')a_2\partial_x^3}F_2(t',\cdot)\,dt'\right\|_{X_{v,\delta}^{\sigma,s,b}} \le C\|F_2\|_{X_{v,\delta}^{\sigma,s,b-1}},$$
for some constant $C>0$ depending only on $b$.
\end{lem}

We now define the sequence $\{(u^{(n)}, v^{(n)})\}_{n=0}^\infty$ as follows:
\[
\begin{cases}
u^{(0)}_t + a_1 u^{(0)}_{xxx} = 0, \\
v^{(0)}_t + a_2 v^{(0)}_{xxx} = 0, \\
u^{(0)}(0) = u_0(x),\quad v^{(0)}(0) = v_0(x),
\end{cases}
\]
and for each $n \in \mathbb{Z}^+$,
\[
\begin{cases}
u^{(n)}_t + a_1 u^{(n)}_{xxx} = c_{11} u^{(n-1)} u^{(n-1)}_x + c_{12} v^{(n-1)} v^{(n-1)}_x, \\
v^{(n)}_t + a_2 v^{(n)}_{xxx} = c_{21} u^{(n-1)}_x v^{(n-1)} + c_{22} u^{(n-1)} v^{(n-1)}_x, \\
u^{(n)}(0,x) = u_0(x),\quad v^{(n)}(0,x) = v_0(x).
\end{cases}
\]

By applying the Duhamel formula \eqref{DF}, we can express each term as
\[
\begin{cases}
u^{(0)}(t,x) = e^{-t a_1 \partial_x^3} u_0(x), \\
v^{(0)}(t,x) = e^{-t a_2 \partial_x^3} v_0(x),
\end{cases}
\]
and for $n \ge 1$,
$$\begin{cases}
u^{(n)}(t,x)=e^{-ta_1\partial_x^3}u_0(x)
-\int_0^t e^{-(t-t')a_1\partial_x^3}\big[-c_{11}u^{(n-1)}u_x^{(n-1)}-c_{12}v^{(n-1)}v_x^{(n-1)}\big](t',\cdot)dt',\\
v^{(n)}(t,x)=e^{-ta_2\partial_x^3}v_0(x)-\int_0^t e^{-(t-t')a_2\partial_x^3}
\big[-c_{21}u_x^{(n-1)}v^{(n-1)}-c_{22}u^{(n-1)}v_x^{(n-1)}\big](t',\cdot)dt'.
\end{cases}$$

By Lemma \ref{lem1} we have
\begin{equation}\label{0step}
\|u^{(0)}\|_{X_{u,\delta}^{\sigma,s,b}}\lesssim_{b}\|u_0\|_{G^{\sigma,s}},\quad \|v^{(0)}\|_{X_{v,\delta}^{\sigma,s,b}}\lesssim_{b}\|v_0\|_{G^{\sigma,s}},
\end{equation}
and 
$$
\|u^{(n)}-u^{(0)}\|_{X_{u,\delta}^{\sigma,s,b}}\lesssim_{b}\big\|-c_{11}u^{(n-1)}u_x^{(n-1)}
-c_{12}v^{(n-1)}v_x^{(n-1)}\big\|_{X_{u,\delta}^{\sigma,s,b-1}}
$$
if $\sigma\ge0,$ $s\in\mathbb R,$ $1/2<b\le1,$  and $0<\delta\leq 1$.
By applying Lemma \ref{lem2} (with $b$ and $b'$ replaced by $b-1$ and $b'-1$ respectively) and then Lemma \ref{bilinear}, the above is further bounded by
\begin{align*}
&\lesssim_{b,b'}\delta^{b'-b}\Big\| -c_{11}u^{(n-1)}u_x^{(n-1)}-c_{12}v^{(n-1)}v_x^{(n-1)}
\Big\|_{X_{u,\delta}^{\sigma,s,b'-1}}\\
&\lesssim_{s,b,b'}\delta^{b'-b}(\|u^{(n-1)} \|_{X_{u,\delta}^{\sigma,s,b}}^2
+\|v^{(n-1)}\|_{X_{v,\delta}^{\sigma,s,b}}^2),
\end{align*}
if $\sigma\ge0,$ $s>-3/4,$ $1/2<b<b^\prime<1$, and $0<\delta\leq1$. 
This yields
\begin{equation*}
\|u^{(n)}\|_{X_{u,\delta}^{\sigma,s,b}}
\lesssim_{s,b,b'}\|u_0\|_{G^{\sigma,s}}
+\delta^{b'-b}\big(\|u^{(n-1)}\|_{X_{u,\delta}^{\sigma,s,b}}^2+\|v^{(n-1)}\|_{X_{v,\delta}^{\sigma,s,b}}^2
\big),
\end{equation*}
and similarly,
\begin{equation*}
\|v^{(n)}\|_{X_{v,\delta}^{\sigma,s,b}}\lesssim_{s,b,b'}\|v_0\|_{G^{\sigma,s}}
+\delta^{b'-b}
\|u^{(n-1)} \|_{X_{u,\delta}^{\sigma,s,b}}\|v^{(n-1)} \|_{X_{v,\delta}^{\sigma,s,b}}, 
\end{equation*}
if $\sigma\ge0,$ $s>-3/4,$ $1/2<b<b^\prime<1$, and $0<\delta\leq1$.  

By induction and the initial bound \eqref{0step}, it follows that for all $n \ge 0$,
\begin{align}
\label{proof1}
\|u^{(n)}\|_{X_{u,\delta}^{\sigma,s,b}} &\lesssim_b \|u_0\|_{G^{\sigma,s}} + \|v_0\|_{G^{\sigma,s}}, \\
\label{proof2}
\|v^{(n)}\|_{X_{v,\delta}^{\sigma,s,b}} &\lesssim_b \|u_0\|_{G^{\sigma,s}} + \|v_0\|_{G^{\sigma,s}},
\end{align}
provided that
\begin{equation}\label{delttt}
\delta^{b'-b} (\|u_0\|_{G^{\sigma,s}} + \|v_0\|_{G^{\sigma,s}}) \ll 1.
\end{equation}

We now observe that
$$
u^{(n)}-u^{(n-1)}=-c_{11}(u^{(n-1)}u_x^{(n-1)}-u^{(n-2)}u_x^{(n-2)})-c_{12}(v^{(n-1)}v_x^{(n-1)}-v^{(n-2)}v_x^{(n-2)}),
$$
and recall the identity $ab - cd = (a - c)b + c(b - d)$.  
Using this decomposition, and applying Lemmas \ref{lem1}, \ref{lem2}, and \ref{bilinear} as before, we obtain 
\begin{align*}
\|u^{(n)}-u^{(n-1)}\|_{X_{u,\delta}^{\sigma,s,b}}\lesssim_{s,b,b'}
\delta^{b'-b}
\Big(&\big\|u^{(n-1)}+u^{(n-2)}\big\|_{X_{u,\delta}^{\sigma,s,b}}\big\|u^{(n-1)}-u^{(n-2)}\big\|_{X_{u,\delta}^{\sigma,s,b}}\\
&+\big\|v^{(n-1)}+v^{(n-2)}\big\|_{X_{v,\delta}^{\sigma,s,b}}\big\|v^{(n-1)}-v^{(n-2)}\big\|_{X_{v,\delta}^{\sigma,s,b}}\Big),
\end{align*}
provided that $\sigma \ge 0$, $s > -3/4$, $1/2 < b < b' < 1$, and $0 < \delta \le 1$.

Next, by applying the uniform bounds \eqref{proof1} and \eqref{proof2}, followed by the smallness condition \eqref{delttt}, we further estimate the above by
\begin{align*}
&\delta^{b'-b}(\|u_0\|_{G^{s,b}}+\|v_0\|_{G^{s,b}})\big(\big\|u^{(n-1)}-u^{(n-2)}\big\|_{X_{u,\delta}^{\sigma,s,b}}+\|v^{(n-1)}-v^{(n-2)}\big\|_{X_{v,\delta}^{\sigma,s,b}}\big)\\
&\qquad\qquad\qquad\le\frac14\big(\big\|u^{(n-1)}-u^{(n-2)}\Big\|_{X_{u,\delta}^{\sigma,s,b}}+\|v^{(n-1)}-v^{(n-2)}\big\|_{X_{v,\delta}^{\sigma,s,b}}\big),
\end{align*}
and hence
$$
\|u^{(n)}-u^{(n-1)}\|_{X_{u,\delta}^{\sigma,s,b}}
\le\frac14\big(\big\|u^{(n-1)}-u^{(n-2)}\Big\|_{X_{u,\delta}^{\sigma,s,b}}+\|v^{(n-1)}-v^{(n-2)}\big\|_{X_{v,\delta}^{\sigma,s,b}}\big).
$$
In the same manner, we obtain
$$\|v^{(n)}-v^{(n-1)}\|_{X_{v,\delta}^{\sigma,s,b}}\le\frac14\big(\big\|u^{(n-1)}-u^{(n-2)}\big\|_{X_{u,\delta}^{\sigma,s,b}}+\|v^{(n-1)}-v^{(n-2)}\big\|_{X_{v,\delta}^{\sigma,s,b}}\big).
$$
These estimates imply that the sequence $\{(u^{(n)}, v^{(n)})\}_{n=0}^\infty$ is Cauchy in the space $X_{u,\delta}^{\sigma,s,b} \times X_{v,\delta}^{\sigma,s,b}$, and hence converges to a solution $(u,v)$ with the desired bounds \eqref{proof1} and \eqref{proof2}, as asserted in \eqref{lowe}.
Finally, the lifespan estimate \eqref{deldel} follows directly from \eqref{delttt} and $0 < b' - b < 1/2$.
\end{proof}

\subsection{Almost conservation law}\label{subsec4.2}
Having established the existence of local-in-time solutions, our next goal is to extend this result to arbitrarily large time intervals. To do so, it is essential to control the growth of the norm  on which the local existence time depends.
In order to apply the local well-posedness result iteratively over successive short time intervals and thereby reach any prescribed final time $T > 0$, we employ an approximate conservation principle. This allows us to propagate the solution forward in time while maintaining control over the norm, by
adjusting the analyticity strip width parameter $\sigma$ according to the size of $T$.

\begin{thm}\label{thm3}
Let $0\le\rho\le1$, $\frac12<b<1$ and $\delta$ be as in Theorem \ref{thm2}. Then there exists $C_b>0$ such that for any $\sigma>0$ and any solution $(u,v)\in\mathcal X^{\sigma,0,b}$ to the Cauchy problem \eqref{ssys} on the time interval $[0,\delta]$, we have 
\begin{equation}\label{acl00}
\sup_{t\in[0,\delta]}\|(u,v)(t)\|^2_{\mathcal G^{\sigma,0}}\le\|(u,v)(0)\|^2_{\mathcal G^{\sigma,0}}+C_b\sigma^\rho\|(u,v)\|^3_{\mathcal X^{\sigma,0,b}_\delta}.
\end{equation}
\end{thm}

\begin{proof}
Let $0 \le \delta' \le \delta$. Define $U(t,x) := e^{\sigma|\partial_x|} u(t,x)$ and $V(t,x) := e^{\sigma|\partial_x|} v(t,x)$. Applying $e^{\sigma|\partial_x|}$ to the system \eqref{ssys}, we obtain
\begin{equation}\label{f1}
U_t + a_1 U_{xxx} = c_{11} U U_x + c_{12} V V_x + f_1(u,v),
\end{equation}
\begin{equation}\label{f2}
V_t + a_2 V_{xxx} = c_{21} U_x V + c_{22} U V_x + f_2(u,v),
\end{equation}
where $f_1$ and $f_2$ are defined as in Lemma \ref{f}.

Multiplying both equations by $U$ and $V$, respectively, and integrating in space, we obtain
$$
\int_\mathbb RUU_tdx+a_1\int_\mathbb RUU_{xxx}dx=c_{11}\int_\mathbb RU^2U_xdx+c_{12}\int_\mathbb RUVV_xdx+\int_\mathbb RUf_1(u,v)dx
$$
and
$$
\int_\mathbb RVV_tdx+a_2\int_\mathbb RVV_{xxx}dx=
c_{21}\int_\mathbb RU_xV^2dx+c_{22}\int_\mathbb RUVV_xdx+\int_\mathbb RVf_2(u,v)dx.
$$
Applying integration by parts to the integrals involving $a_1,a_2$, or $c_{12}$, we obtain
$$
\frac12\int_\mathbb R(U^2)_tdx-\frac{a_1}2\int_\mathbb R(U_x^2)_xdx
=\frac{c_{11}}3\int_\mathbb R(U^3)_xdx+c_{12}\int_\mathbb RUVV_xdx+\int_\mathbb RUf_1(u,v)dx
$$
and
$$
\frac12\int_\mathbb R(V^2)_tdx-\frac{a_2}2\int_\mathbb R(V_x^2)_xdx
=-2c_{21}\int_\mathbb RUVV_xdx+c_{22}\int_\mathbb RUVV_xdx+\int_\mathbb RVf_2(u,v)dx.
$$
Here we may assume that $U,V$, and all their spatial derivatives decay to zero as $|x|\rightarrow \infty$.\footnote{This property can be shown by approximation using the monotone convergence theorem and the Riemann-Lebesgue lemma. See the argument in \cite{SS}, p. 1018.}
Under this assumption, the above equations reduce to
$$
\frac12\int_\mathbb R(U^2)_tdx
=c_{12}\int_\mathbb RUVV_xdx+\int_\mathbb RUf_1(u,v)dx
$$
and
$$
\frac12\int_\mathbb R(V^2)_tdx
=(c_{22}-2c_{21})\int_\mathbb RUVV_xdx+\int_\mathbb RVf_2(u,v)dx.
$$

Now, suppose the coefficients satisfy  
$c_{12} = \eta(c_{22} - 2c_{21})$ for some constant $\eta$.\footnote{For the Majda–Biello system, we take $\eta = -1$; for the Hirota–Satsuma system, we take $\eta = -\frac{c_{12}}{3}$.}
Then, multiplying the second equation by $\eta$ and summing the two equations yields
$$\frac12\int_\mathbb R(U^2)_tdx+\frac\eta2\int_\mathbb R(V^2)_tdx=\int_\mathbb R(Uf_1(u,v)+\eta Vf_2(u,v))dx.$$
Integrating in time over the interval $[0,\delta']$, we obtain
\begin{align*}
&\|u(\delta')\|_{G^{\sigma,0}}^2+\eta\|v(\delta')\|_{G^{\sigma,0}}^2\\&\qquad\le\|u(0)\|_{G^{\sigma,0}}^2+\eta\|v(0)\|_{G^{\sigma,0}}^2+2\left|\int_{\mathbb R^{1+1}}\chi_{[0,\delta']}(t)(Uf_1(u,v)+\eta Vf_2(u,v))dtdx\right|.
\end{align*}
Using H\"older's inequality, together with Lemmas \ref{lem3} and \ref{f}, we estimate the integral term as follows:
\begin{align*}
\Big|\int_{\mathbb R^{1+1}}&\chi_{[0,\delta']}(t)(Uf_1(u,v)+\eta Vf_2(u,v))dtdx\Big|\\
\le&\|\chi_{[0,\delta']}(t)U\|_{X_u^{0,1-b}}\|\chi_{[0,\delta']}(t)f_1(u,v)\|_{X_u^{0,b-1}}+\eta\|\chi_{[0,\delta']}(t)V\|_{X_v^{0,1-b}}\|\chi_{[0,\delta']}(t)f_2(u,v)\|_{X_v^{0,b-1}}\\
\lesssim_b&\Big(\|U\|_{X_{u,\delta'}^{0,1-b}}\|f_1(u,v)\|_{X_{u,\delta'}^{0,b-1}}+\eta\|V\|_{X_{v,\delta'}^{0,1-b}}\|f_2(u,v)\|_{X_{v,\delta'}^{0,b-1}}\Big)\\
\lesssim_b&\|(u,v)\|_{\mathcal X_{\delta'}^{\sigma,0,1-b}}\sigma^\rho\|(u,v)\|^2_{\mathcal X_{\delta'}^{\sigma,0,b}}.
\end{align*}
Since $1-b<b$, we conclude that
$$\sup_{t\in[0,\delta]}\|(u,v)(t)\|_{\mathcal G^{\sigma,0}}^2\lesssim_b\|(u,v)(0)\|_{\mathcal G^{\sigma,0}}^2+\sigma^\rho\|(u,v)\|_{\mathcal X_{\delta'}^{\sigma,0,b}}^3,$$
as desired.
\end{proof}

\section{Proof of Theorem \ref{thm1}}\label{sec5}
A system of the form \eqref{ssys} is invariant under the reflection $(t, x) \mapsto (-t, -x)$. Hence, it suffices to consider positive times only. By the embedding \eqref{emb}, the general case $s \in \mathbb{R}$ will be reduced to the case $s = 0$ at the end of this section.

\subsection{The case $s=0$}
Combining \eqref{acl00} and \eqref{lowe}, we first observe that
\begin{equation}\label{acl01}
\sup_{t \in [0, \delta]} \|(u,v)(t)\|_{\mathcal{G}^{\sigma,0}}^2 
\le \|(u,v)(0)\|_{\mathcal{G}^{\sigma,0}}^2 + C \sigma^\rho \|(u,v)(0)\|_{\mathcal{G}^{\sigma,0}}^3.
\end{equation}

Let $(u_0, v_0) = (u,v)(0) \in \mathcal{G}^{\sigma_0, 0}$ for some $\sigma_0 > 0$, and let $\delta$ be as in Theorem \ref{thm2}. For any arbitrarily large $T > 0$, our goal is to show that the solution $(u,v)$ to \eqref{ssys} satisfies
$$(u,v)(t) \in \mathcal{G}^{\sigma(t), 0} \quad \text{for all } t \in [0, T],$$
where
\begin{equation}\label{sigma}
\sigma(t) \ge \frac{c}{T^{4/3 + \epsilon}}
\end{equation}
for some constant $c > 0$ depending on $\|(u_0, v_0)\|_{\mathcal{G}^{\sigma_0, 0}}$ and $\sigma_0$.
To establish this, fix $T > 0$ arbitrarily. It suffices to show that
\begin{equation}\label{label}
\sup_{t \in [0, T]} \|(u,v)(t)\|_{\mathcal{G}^{\sigma,0}}^2 
\le 2 \|(u_0, v_0)\|_{\mathcal{G}^{\sigma_0, 0}}^2
\end{equation}
for some $\sigma$ satisfying \eqref{sigma}. This will imply that $(u,v)(t) \in \mathcal{G}^{\sigma(t), 0}$ for all $t \in [0, T]$, as desired.

To prove \eqref{label}, we choose $n \in \mathbb{Z}^+$ such that $n\delta \le T \le (n+1)\delta$. Using induction, we will show that for each $k \in \{1,2,\ldots,n+1\}$,
\begin{equation}\label{label1}
\sup_{t \in [0, k\delta]} \|(u,v)(t)\|_{\mathcal{G}^{\sigma,0}}^2 
\le \|(u_0, v_0)\|_{\mathcal{G}^{\sigma,0}}^2 + k C \sigma^\rho 2^{3/2} \|(u_0, v_0)\|_{\mathcal{G}^{\sigma_0, 0}}^3,
\end{equation}
and
\begin{equation}\label{label2}
\sup_{t \in [0, k\delta]} \|(u,v)(t)\|_{\mathcal{G}^{\sigma,0}}^2 
\le 2 \|(u_0, v_0)\|_{\mathcal{G}^{\sigma_0, 0}}^2,
\end{equation}
provided $\sigma$ satisfies
\begin{equation}\label{label3}
\sigma \le \sigma_0 
\quad \text{and} \quad 
\frac{2T}{\delta} C \sigma^\rho 2^{3/2} \|(u_0, v_0)\|_{\mathcal{G}^{\sigma_0, 0}} \le 1.
\end{equation}

We begin with the base case $k = 1$. From \eqref{acl01}, we have
\begin{align*}
\sup_{t \in [0, \delta]} \|(u,v)(t)\|_{\mathcal{G}^{\sigma,0}}^2 
&\le \|(u_0, v_0)\|_{\mathcal{G}^{\sigma,0}}^2 
+ C \sigma^\rho \|(u_0, v_0)\|_{\mathcal{G}^{\sigma,0}}^3 \\
&\le 2 \|(u_0, v_0)\|_{\mathcal{G}^{\sigma_0, 0}}^2,
\end{align*}
where we used $\|(u_0, v_0)\|_{\mathcal{G}^{\sigma,0}} \le \|(u_0, v_0)\|_{\mathcal{G}^{\sigma_0,0}}$ and $C \sigma^\rho \|(u_0, v_0)\|_{\mathcal{G}^{\sigma_0,0}} \le 1$, which follow from \eqref{label3}.

Now suppose that \eqref{label1} and \eqref{label2} hold for some $k \in \{1,2,\ldots,n\}$. Applying \eqref{acl01} at $t = k\delta$, and using the induction hypotheses \eqref{label2} and \eqref{label1}, we obtain
\begin{align*}
\sup_{t \in [k\delta, (k+1)\delta]} \|(u(t), v(t))\|_{\mathcal{G}^{\sigma,0}}^2 
&\le \|(u,v)(k\delta)\|_{\mathcal{G}^{\sigma,0}}^2 
+ C \sigma^\rho \|(u,v)(k\delta)\|_{\mathcal{G}^{\sigma,0}}^3 \\
&\le \|(u,v)(k\delta)\|_{\mathcal{G}^{\sigma,0}}^2 
+ C \sigma^\rho 2^{3/2} \|(u_0, v_0)\|_{\mathcal{G}^{\sigma_0,0}}^3 \\
&\le \|(u_0, v_0)\|_{\mathcal{G}^{\sigma,0}}^2 
+ C \sigma^\rho (k+1) 2^{3/2} \|(u_0, v_0)\|_{\mathcal{G}^{\sigma_0,0}}^3.
\end{align*}
Combining this with the induction hypothesis \eqref{label1}, we deduce
\begin{equation}\label{label4}
\sup_{t \in [0, (k+1)\delta]} \|(u,v)(t)\|_{\mathcal{G}^{\sigma,0}}^2 
\le \|(u_0, v_0)\|_{\mathcal{G}^{\sigma,0}}^2 
+ C \sigma^\rho (k+1) 2^{3/2} \|(u_0, v_0)\|_{\mathcal{G}^{\sigma_0,0}}^3,
\end{equation}
which proves \eqref{label1} for $k+1$.
Moreover, since $k+1 \le n+1 \le T/\delta + 1 \le 2T/\delta$, the condition \eqref{label3} implies
\[
C \sigma^\rho (k+1) 2^{3/2} \|(u_0, v_0)\|_{\mathcal{G}^{\sigma_0,0}} 
\le \frac{2T}{\delta} C \sigma^\rho 2^{3/2} \|(u_0, v_0)\|_{\mathcal{G}^{\sigma_0,0}} 
\le 1,
\]
and hence, by \eqref{label4}, the bound \eqref{label2} holds for $k+1$.

Finally, the condition \eqref{label3} is satisfied by choosing
\[
\sigma = \left( \frac{\delta}{C 2^{5/2} \|(u_0, v_0)\|_{\mathcal{G}^{\sigma_0,0}}} \right)^{1/\rho} \cdot T^{-1/\rho}.
\]
Since we are free to choose any $0 \le \rho < 3/4$, we obtain \eqref{sigma} with
\[
c = \left( \frac{\delta}{C 2^{5/2} \|(u_0, v_0)\|_{\mathcal{G}^{\sigma_0,0}}} \right)^{1/\rho},
\]
which depends only on $\|(u_0, v_0)\|_{\mathcal{G}^{\sigma_0,0}}$.

\subsection{The general case $s \in \mathbb{R}$}
Recall from the embedding \eqref{emb} that
\[
\mathcal{G}^{\sigma, s} \subset \mathcal{G}^{\sigma', s'} \quad \text{for all } \sigma > \sigma' \ge 0 \text{ and } s, s' \in \mathbb{R}.
\]
In particular, for any $s \in \mathbb{R}$, we have
\[
(u_0, v_0) \in \mathcal{G}^{\sigma_0, s} \subset \mathcal{G}^{\sigma_0/2, 0}.
\]
From the local well-posedness result, it follows that there exists $\delta = \delta(\|(u_0, v_0)\|_{\mathcal{G}^{\sigma_0/2, 0}})$ such that
\[
(u,v)(t) \in \mathcal{G}^{\sigma_0/2, 0} \quad \text{for } 0 \le t \le \delta.
\]
Arguing as in the case $s = 0$, for any $T > \delta$, we obtain
\[
(u,v)(t) \in \mathcal{G}^{\sigma', 0} \quad \text{for } t \in [0, T],
\]
where $\sigma' \ge c/T$ for some constant $c > 0$ depending on $\|(u_0, v_0)\|_{\mathcal{G}^{\sigma_0/2, 0}}$ and $\sigma_0$.
Applying the embedding once more, we conclude that
\[
(u,v)(t) \in \mathcal{G}^{\sigma, s} \quad \text{for } t \in [0, T],
\]
where $\sigma = \sigma'/2$.
\qed

\section*{Acknowledgments}
The authors thank the anonymous referees for their careful reading of the manuscript and for their helpful comments, which have helped improve the presentation of the paper.

\end{document}